\documentclass{amsart}[12pt]
\usepackage{amsmath, amsthm, amscd, amsfonts,}
\usepackage{graphicx,float}

\usepackage{amssymb}

\usepackage{pb-diagram}
\newcommand{\url}[1]{{\tt #1}}

\def\k{\kappa}

\newtheorem{theorem}{Theorem}[section]
\newtheorem{lemma}[theorem]{Lemma}

\newtheorem{open Question}[theorem]{Open Question}

\newtheorem{remark}[theorem]{Remark}

\numberwithin{equation}{section}

\def\k{\kappa}

\def\rmark{\mbox{$\rm\bf\rule{0.06em}{1.45ex}\kern-0.05em R$}}
\def\pmark{\mbox{$\rm\bf\rule{0.06em}{1.45ex}\kern-0.05em P$}}
\def\nmark{\mbox{$\rm\bf\rule{0.06em}{1.45ex}\kern-0.05em N$}}
\def\vdash{\mbox{$\rm\| \kern-0.13em -$}}

\usepackage{pb-diagram}

\def\rmark{\mbox{$\rm\bf\rule{0.06em}{1.45ex}\kern-0.05em R$}}
\def\pmark{\mbox{$\rm\bf\rule{0.06em}{1.45ex}\kern-0.05em P$}}
\def\nmark{\mbox{$\rm\bf\rule{0.06em}{1.45ex}\kern-0.05em N$}}
\def\vdash{\mbox{$\rm\| \kern-0.13em -$}}

\begin{document}

\title[Singular cofinality conjecture and a question of Gorelic]{Singular cofinality conjecture and a question of Gorelic }

\author[M. Golshani.]{Mohammad Golshani}

\thanks{The author's research was in part supported by a grant from IPM (No. 91030417).}

\maketitle

\begin{abstract}
We give an affirmative answer to a question of Gorelic \cite{Gorelic}, by showing it is consistent, relative to the existence of large cardinals, that there is a proper class of cardinals $\alpha$ with $cf(\alpha)=\omega_1$ and $\alpha^\omega > \alpha.$
\end{abstract}

\section{introduction}
Around 1980, Pouzet \cite{pouzet} proved the fundamental result that if $(\mathbb{P}, \leq)$ is a poset of singular cofinality, then it contains an infinite antichain. This lead to the formulation of a very natural conjecture, first appearing
implicitly in \cite{pouzet}, and then explicitly in a paper by Milner and Sauer \cite{milner-sauer}:
\\
{\bf Conjecture.} Suppose that $(\mathbb{P}, \leq)$ is a poset of singular cofinality $\lambda.$ Then $(\mathbb{P}, \leq)$ has an antichain of size $cf(\lambda).$

This is called the \emph{Singular Cofinality Conjecture}.

Set  $C=\{\alpha: \alpha$ is a cardinal, $cf(\alpha)=\omega_1, \alpha^\omega >\alpha \}.$ In \cite{Gorelic}, Gorelic observed that if $C$ is not a proper class, then the Singular Cofinality Conjecture holds "ultimately"
(in ZFC) in the case of cofinality $\omega_1$, and he
asked if it is consistent that $C$ is a proper class.  In this paper we give an affirmative answer to this question, assuming the existence of large cardinals:

\begin{theorem}
Assuming the existence of suitable large cardinals, it is consistent that
$C=\{\alpha: \alpha$ is a cardinal, $cf(\alpha)=\omega_1, \alpha^\omega >\alpha \}$ is a proper class.
\end{theorem}
\begin{remark}
We give three different proofs for the above theorem. The first proof uses a strong cardinal ( in fact a $\kappa^{+\omega_1+2}-$strong cardinal $\kappa$) and is based on extender based Radin forcing. The second proof  assumes the existence of a proper class of $\kappa^{+\omega_1+1}-$strong cardinals $\kappa$, and is based on iterated Prikry forcing. The third proof also assumes the existence of a proper class of $\kappa^{+\omega_1+1}-$strong cardinals $\kappa$, and is based on iteration of extender based Prikry forcing. We also show that the large cardinal assumption in our second and third proofs is almost optimal.
\end{remark}

\section{proof of the main theorem}
\subsection{First proof}
In this subsection we give our  first proof of the main Theorem 1.1., assuming the existence of a strong cardinal. Thus suppose that $GCH$ holds and let $\kappa$ be a strong cardinal. Let $j: V \rightarrow M$ be an elementary embedding of the universe into some inner model $M$ with $crit(j)=\kappa$ and $M \supseteq V_{\kappa^{+\omega_1+2}}.$ Using $j$ construct, as in \cite{merimovich}, an extender sequence system $\bar{E}$ of length $\kappa^+$ and of size $\kappa^{+\omega_1+1},$ and let $\mathbb{P}_{\bar{E}}$ be the corresponding extender based Radin frocing as is defined in \cite{merimovich}. Also let $G$ be $\mathbb{P}_{\bar{E}}-$generic over $V$. Then:

\begin{theorem}( \cite{merimovich})
$(a)$ $V$ and $V[G]$ have the same cardinals,

$(b)$ $\kappa$ remains an inaccessible cardinal in $V[G],$

$(c)$ In $V[G],$ there exists a club $\bar{C}$ of $\kappa,$ such that if $\gamma$ is a limit point of $\bar{C},$ then $2^\gamma = \gamma^{+\omega_1+1}$
\end{theorem}
 By $(b)$, $V_\kappa$ of $V[G]$ is a model of $ZFC$, and the following lemma shows that in it, $C$ is a proper class, which completes the proof of Theorem 1.1.
\begin{lemma}
In $V[G], \{\alpha <\kappa: \alpha$ is a cardinal, $cf(\alpha)=\omega_1, \alpha^\omega >\alpha \} \supseteq \{\gamma^{+\omega_1}: \gamma$ is a limit point of $\bar{C}, cf(\gamma)=\omega \}.$
\end{lemma}
\begin{proof}
Suppose $\gamma$ is a limit point of $\bar{C}$ of cofinality $\omega.$ Then clearly $cf(\gamma^{+\omega_1})=\omega_1.$ We also have $(\gamma^{+\omega_1})^\omega \geq \gamma^\omega=2^\gamma=  \gamma^{+\omega_1+1} >\gamma^{+\omega_1}.$
\end{proof}
\subsection{Second proof}
We now give our second proof of the main Theorem 1.1., assuming the existence of a proper class of $\kappa^{+\omega_1+1}-$strong cardinals $\kappa$. Thus assume $GCH$ holds and suppose that there exists a proper class $A$ of $\kappa^{+\omega_1+1}-$strong cardinals $\kappa$. We may assume that no element of $A$ is a limit point of $A$.

{\bf Step 1)} Let $\mathbb{P}$ be the reverse Easton iteration of $Sacks(\alpha, \alpha^{+\omega_1+1})$ for each inaccessible cardinal $\alpha$, and let $G$ be $\mathbb{P}-$generic over $V$.  Then:

\begin{theorem} ( \cite{friedman-honzik})
$(a)$ $V$ and $V[G]$ have the same cardinals and cofinalities,

$(b)$ $V[G] \models ``$for each inaccessible cardinal $\alpha, 2^{\alpha}=\alpha^{+\omega_1+1}$'',

$(c)$ Each $\alpha\in A$ is measurable in $V[G].$
\end{theorem}

{\bf Step 2)} Working in $V[G],$ let $\mathbb{Q}$ be the forcing defined in [1, $\S$3.1], for changing the cofinality of each $\alpha\in A$ to $\omega,$ and let $H$ be $\mathbb{Q}-$generic over $V[G].$

\begin{theorem} ( \cite{friedman-golshani})
$(a)$ $V[G]$ and $V[G][H]$ have the same cardinals,

$(b)$ For each $\alpha\in A, V[G][H]\models ``\alpha$ is a strong limit cardinal, $cf(\alpha)=\omega$ and $2^{\alpha}=\alpha^{+\omega_1+1}$''.
\end{theorem}
The following lemma completes the proof of the theorem:
\begin{lemma}
In $V[G][H], C \supseteq \{\alpha^{+\omega_1}: \alpha\in A\}$.
\end{lemma}
\begin{proof}
Work in $V[G][H]$ and let $\alpha\in A.$   Clearly $cf(\alpha^{+\omega_1})=\omega_1.$ We also have
$(\alpha^{+\omega_1})^\omega \geq \alpha^\omega=2^\alpha=\alpha^{+\omega_1+1} > \alpha^{+\omega_1}.$
\end{proof}
\subsection{Third proof}
In  this subsection we give our third proof of the main Theorem 1.1., assuming the existence of a proper class of $\kappa^{+\omega_1+1}-$strong cardinals $\kappa$. Again assume $GCH$ holds and let $A$ be a proper class of $\kappa^{+\omega_1+1}-$strong cardinals $\kappa$, such  that no element of $A$ is a limit point of $A$.

 For each $\kappa\in A,$ fix a $(\kappa, \kappa^{+\omega_1+1})-$extender $E(\k)$ and let $(\mathbb{P}_{E(\k)}, \leq_{\mathbb{P}_{E(\k)}}, \leq^*_{\mathbb{P}_{E(\k)}})$ (where $\leq^*_{\mathbb{P}_{E(\k)}}$ is the Prikry extension relation) be the corresponding extender based Prikry forcing for changing the cofinality of $\kappa$ into $\omega,$ and making $2^\k=\kappa^{+\omega_1+1}$ \cite{gitik-magidor}.

 Let $\mathbb{P}$ be the following version of iterated extender based Prikry forcing. Conditions in $\mathbb{P}$ are of the form $p=(X^p, F^p),$ where

 \begin{enumerate}
\item $X^p$ is a subset of $A$,
\item $F^p$ is a function on $X^p$,
\item For all $\k\in X^p, F^p(\k)$ is a condition in $\mathbb{P}_{E(\k)}$.
\end{enumerate}
Given $p, q\in \mathbb{P},$ we define $p\leq q$ ($p$ is stronger than $q$), if
 \begin{enumerate}
\item $X^p \supseteq X^q,$
\item For all $\k\in X^q,$ $F^p(\k) \leq_{\mathbb{P}_{E(\k)}} F^q(\k)$.
\end{enumerate}
We also define the Prikry relation by $p\leq^* q$ iff
\begin{enumerate}
\item $p\leq q,$
\item For all $\k\in X^q,$ $F^p(\k) \leq^*_{\mathbb{P}_{E(\k)}} F^q(\k)$.
\end{enumerate}
Let $G$ be $\mathbb{P}-$generic over $V$.
Then using the methods of \cite{friedman-golshani} and \cite{gitik-magidor} we can prove the following:
\begin{theorem}
$(a)$ $\mathbb{P}$ is a tame class forcing notion; in particular $V[G]\models ZFC,$

$(b)$ $(\mathbb{P}, \leq, \leq^*)$ satisfies the Prikry property,

$(c)$ $V$ and $V[G]$ have the same cardinals,

$(d)$ For each $\alpha\in A, V[G]\models ``\alpha$ is a strong limit cardinal, $cf(\alpha)=\omega$ and $2^{\alpha}=\alpha^{+\omega_1+1}$''.
\end{theorem}
The rest of the argument is as in the second proof.
\section{necessary use of large cardinals}

In this section we show that some large cardinal assumptions are needed for the proof of  Theorem 1.1.
\begin{theorem}
Assume there is a model $V$ of $ZFC$ in which $C$ is a proper class. Then there is an inner model of $ZFC$ which contains a proper class of measurable cardinals.
\end{theorem}
\begin{proof}
We may assume that  there is no inner model with a strong cardinal, as otherwise we are done. Let $\mathcal{K}$ denote the core model of $V$ below a strong cardinal. Assume on the contrary that the measurable cardinals of $\mathcal{K}$ are bounded, say by $\lambda > 2^{\omega_1}.$ Then for all $\alpha >2^\lambda$ with $cf(\alpha)=\omega_1,$ we have
\[ [\alpha]^{\omega}= \bigcup_{\delta<\alpha}[\delta]^\omega.  \]
On the other hand,  by the covering lemma,
\[ [\delta]^\omega \subseteq [\delta]^{\leq\lambda} \subseteq \bigcup_{x\in \mathcal{K}\cap [\delta]^\lambda}P(x), \]
and hence
\[ \delta^\omega \leq \sum_{x\in \mathcal{K}\cap [\delta]^\lambda}|P(x)|\leq |\mathcal{K}\cap [\delta]^\lambda|.2^\lambda\leq\delta^+.2^\lambda <\alpha,   \]
which implies
\[ \alpha^\omega=\alpha \]
Thus $C \subseteq (2^{\lambda})^+$ is bounded, and we get a contradiction.
\end{proof}
In fact we can prove more:
\begin{theorem}
Assume there is a model $V$ of $ZFC$ in which $C$ is a proper class. Then $\{\delta: \delta$ is a cardinal, $cf(\delta)=\omega$ and $\delta^\omega \geq \delta^{+\omega_1+1} \}$ is a proper class.
\end{theorem}
\begin{proof}
Given $2^\omega <\alpha\in C,$ we have $cf(\alpha)=\omega_1$ and $\alpha^\omega \geq \alpha^+,$ hence there is $\gamma <\alpha$ such that $\gamma^\omega \geq \alpha^+.$ Let $\delta$ be a singular cardinal of cofinality $\omega$ in the interval $(\gamma, \alpha).$ Then $\delta^\omega \geq \alpha^+ \geq \delta^{+\omega_1+1}.$
\end{proof}
It follows from the above theorem and the results of \cite{gitik-mitchell} that the large cardinal assumption made in our second and third proofs is almost optimal.
\section{a generalization}
In general, for an infinite cardinal $\lambda,$ set $C_\lambda=\{\alpha: \alpha$ is a cardinal, $cf(\alpha)=\lambda^+$ and $\alpha^{\lambda} >\alpha^{<\lambda}=\alpha \}.$ Then by a simple modification of the above proofs we have the following:
\begin{theorem}
Suppose $GCH$ holds, $\kappa$ is a strong cardinal and $\lambda$ is an infinite cardinal less than $\kappa.$ Then there is a cardinal preserving generic extension $V[G]$ of the universe in which $\kappa$ remains inaccessible, no new subsets of $\lambda^+$ are added (in particular it remains regular in the extension), and $C_\lambda \cap \kappa$ is unbounded in $\kappa$.
\end{theorem}

\begin{theorem}
Suppose $GCH$ holds, $\lambda$ is an infinite cardinal, and there exists a proper class of $\kappa^{+\lambda^++1}-$strong cardinals $\kappa.$ Then there is a generic extension $V[G]$ of $V$ in which $C_\lambda$ is a proper class.
\end{theorem}

School of Mathematics, Institute for Research in Fundamental Sciences (IPM), P.O. Box:
19395-5746, Tehran-Iran.

E-mail address: golshani.m@gmail.com

\end{document}